\documentclass[a4,12pt,reqno]{amsart}

\usepackage{amsfonts,amsmath,amsthm}
\usepackage{amssymb,epsfig}
\usepackage{fancybox}
\usepackage[centertags]{amsmath}
\usepackage{graphicx}
\usepackage{a4wide}

\usepackage{cases}
\usepackage[usenames]{color}
\usepackage{enumerate} 

\usepackage{color} 
\definecolor{vert}{rgb}{0,0.6,0}

 \theoremstyle{plain}
 \begingroup
 \newtheorem{thm}{Theorem}[section]
  
 \newtheorem{lem}[thm]{Lemma}
 \newtheorem{prop}[thm]{Proposition}
 
 \endgroup

 \theoremstyle{definition}
 \begingroup

 \endgroup

 \theoremstyle{remark}
 \begingroup
 \newtheorem{rem}[thm]{Remark}
 \endgroup

 \numberwithin{equation}{section}

\newcommand{\N}{\mathbb{N}}
\newcommand{\R}{\mathbb{R}}

\newcommand{\T}{\mathbb{T}}

\newcommand{\cA}{\mathcal{A}}

\newcommand{\cC}{\mathcal{C}}

\newcommand{\cM}{\mathcal{M}}

\newcommand{\BUC}{{\rm BUC\,}}
\newcommand{\USC}{{\rm USC\,}}
\newcommand{\LSC}{{\rm LSC\,}}
\newcommand{\Li}{L^{\infty}}
\newcommand{\W}{W^{1,\infty}}

\newcommand{\Q}{\mathbb{T}^{n}\times(0,\infty)}

\newcommand{\cQ}{\mathbb{T}^{n}\times[0,\infty)}

\newcommand{\al}{\alpha}
\newcommand{\gam}{\gamma}
\newcommand{\del}{\delta}
\newcommand{\ep}{\varepsilon}

\newcommand{\lam}{\lambda}
\newcommand{\sig}{\sigma}
\newcommand{\om}{\omega}

\newcommand{\ol}{\overline}
\newcommand{\ul}{\underline}

\newcommand{\dist}{{\rm dist}\,}

\newcommand{\sgn}{{\rm sgn}\,}

\def\limssup{\mathop{\rm limsup\!^*}}
\def\limiinf{\mathop{\rm liminf_*}}

\begin{document}
\title[Weakly Coupled Systems of Hamilton--Jacobi Equations]
{
Remarks on the large time behavior of  
viscosity solutions of quasi-monotone weakly coupled systems 
of Hamilton--Jacobi equations
}

\author[H. MITAKE]
{Hiroyoshi MITAKE}
\author[H. V. TRAN]
{Hung V. Tran}

\address[H. Mitake]
{Department of Applied Mathematics,
Graduate School of Engineering
Hiroshima University
Higashi-Hiroshima 739-8527, Japan}
\email{mitake@hiroshima-u.ac.jp}

\address[H. V. Tran]
{Department of Mathematics, 
University of California, Berkeley, CA 94720, USA}
\email{tvhung@math.berkeley.edu}


\thanks{The work of H. M is partly supported by Research Fellowship for Young Researcher from JSPS, No. 22-1725.}

\date{\today}

\begin{abstract}
We investigate the large-time behavior of viscosity solutions 
of quasi-monotone weakly coupled systems of Hamilton--Jacobi equations 
on the $n$-dimensional torus. 
We establish a convergence result to asymptotic solutions as time goes 
to infinity under rather restricted assumptions. 
\end{abstract}

\maketitle


\section{Introduction}
In this paper 
we study the large time behavior of
the viscosity solutions of
the following weakly coupled 
systems of Hamilton--Jacobi equations 
\begin{numcases}
{\textrm{(C)} \hspace{1cm}}
(u_{1})_t + H_{1}(x,Du_{1}) + c_1(u_{1}-u_{2}) = 0
& in $\Q$, \nonumber \\
(u_{2})_t + H_{2}(x,Du_{2}) + c_2(u_{2}-u_{1}) = 0
& in $\Q$, \nonumber \\
u_{1}(x,0)=u_{01}(x), \ 
u_{2}(x,0)=u_{02}(x)
& 
on $\T^{n}$, 
\nonumber
\end{numcases}
where 
the Hamiltonians $H_{i} \in C(\T^n \times \R^n)$ are 
given functions which are assumed to be 
 \textit{coercive}, i.e., 
\begin{itemize}
\item[{\rm(A1)}]
\hspace{1cm} $\displaystyle
\lim_{r\to\infty}\inf\{H_{i}(x,p)\mid x\in\T^{n}, 
|p|\ge r \}=\infty$, 
\end{itemize}
and $u_{0i}$ 
are given real-valued continuous functions 
on  $\T^{n}$, and 
$c_{i}>0$ are given constants
for $i=1,2$, respectively. 
Here $u_{i}$ are the real-valued unknown functions on $\cQ$ and 
$(u_{i})_t:=\partial u_{i}/\partial t, Du_{i}:=(\partial u_{i}/\partial x_1,\ldots,\partial u_{i}/\partial x_n)$ for $i=1,2$, respectively. 
For the sake of simplicity,
we focus on the system 
of two equations above 
in Cases 1, 2 below but we can easily generalize it 
to general systems of $m$ equations. 
We are dealing only with viscosity solutions of Hamilton--Jacobi equations 
in this paper and thus the term ``viscosity" may be omitted henceforth.

Although it is already established well that 
existence and uniqueness results for weakly coupled systems 
of Hamilton--Jacobi equations hold 
(see \cite{LEN88, EL, IK1} and the references therein for instance),  
there are not many studies on properties of solutions of (C). 
Recently F. Camilli, O. Ley and P. Loreti \cite{CLL} 
investigated homogenization problems for the system 
and obtained the convergence result, and 
the second author with F. Cagnetti and D. Gomes  \cite{CGT2}
considered new nonlinear adjoint methods for weakly coupled systems 
of stationary Hamilton--Jacobi equations 
and obtained the speed of convergence by using usual 
regularized equations. 
As far as the authors know, there are few works on the large time behavior 
of solutions of weakly coupled systems of 
Hamilton--Jacobi equations.

%
%

\subsection{Heuristic derivations and Main goal}
First we heuristically derive  
the large time asymptotics for (C). 
For simplicity, from now on, we assume that
$c_1=c_2=1$. 
We consider the formal asymptotic expansions of the solutions
$u_1, u_2$ of (C) of the form 
\begin{align}
u_1(x,t)&=a_{01}(x)t+a_{11}(x)+a_{21}(x)t^{-1}+\ldots, \notag\\
u_2(x,t)&=a_{02}(x)t+a_{12}(x)+a_{22}(x)t^{-1}+\ldots \quad\mbox{as }\ t\to \infty.\notag
\end{align}
Plugging these expansions into (C), we get
\begin{align}
a_{01}(x)-a_{21}(x)t^{-2}+\ldots+H_1(x,Da_{01}(x)t+Da_{11}(x)+Da_{21}(x)t^{-1}+\ldots)\notag\\
+(a_{01}(x)-a_{02}(x))t+(a_{11}(x)-a_{12}(x))+(a_{21}(x)-a_{22}(x))t^{-1}+\ldots=0, \label{heu-1}
\end{align}
and
\begin{align}
a_{02}(x)-a_{22}(x)t^{-2}+\ldots+H_2(x,Da_{02}(x)t+Da_{12}(x)+Da_{22}(x)t^{-1}+\ldots)\notag\\
+(a_{02}(x)-a_{01}(x))t+(a_{12}(x)-a_{11}(x))+(a_{22}(x)-a_{21}(x))t^{-1}+\ldots=0. \label{heu-2}
\end{align}
Adding up the two equations above, we have 
$$
H_1(x,D a_{01} t + D a_{11} + O(1/t)) + H_2(x,D a_{02} t + D a_{12} + O(1/t)) +O(1)=0
$$
as $t\to\infty$. 
Therefore by the coercivity of $H_1$ and $H_2$ 
we formally get $D a_{01} = Da_{02} \equiv 0$. 
Then sending $t \to \infty$ in \eqref{heu-1}, \eqref{heu-2}, we derive 
$$
a_{01}(x)=a_{02}(x) \equiv a_{0} \ \mbox{for some constant }\, a_{0},
$$
and
\begin{numcases}
{}
H_{1}(x,Da_{11}(x))+a_{11}(x)-a_{12}(x)=-a_{0} 
& in  $\T^{n}$,\notag\\
H_{2}(x,Da_{12}(x))+a_{12}(x)-a_{11}(x)=-a_{0}
& in  $\T^{n}$. \notag
\end{numcases}

Therefore it is natural to investigate the existence of 
solutions of 
\begin{numcases}
{{\rm (E)} \hspace{1cm}}
H_{1}(x,Dv_{1}(x))+v_{1}-v_{2}= c 
& in  $\T^{n}$,\ \nonumber\\
H_{2}(x,Dv_{2}(x))+v_{2}-v_{1}= c 
& in  $\T^{n}$.  \nonumber
\end{numcases}
Here one seeks for a triplet 
$(v_{1},v_{2},c)\in C(\T^n)^{2}\times\R$ 
such that $(v_{1},v_{2})$ is a solution of (E). 
If $(v_{1},v_{2},c)$ is such a triplet, 
we call $(v_{1},v_{2})$ a \textit{pair of ergodic functions} 
and $c$ an \textit{ergodic constant}. 
By an analogous argument to that of the classical result 
of \cite{LPV} we can see that there exists a solution 
of (E). 
Indeed the second author with F. Cagnetti, D. Gomes \cite{CGT2} 
recently proved that there exists a unique constant $c$ 
such that the ergodic problem has continuous solutions $(v_1, v_2)$.

Hence, our goal in this paper is to prove the following 
large time asymptotics for (C)
under appropriate assumptions on $H_{i}$. 
For any $(u_{01}, u_{02})\in C(\T^{n})^{2}$ 
there exists a solution 
$(v_{1},v_{2},c)\in C(\T^{n})^{2}\times\R$ 
of {\rm (E)} 
such that if $(u_{1}, u_{2})\in C(\cQ)^{2}$ is 
the solution of {\rm (C)}, 
then, as $t\to\infty$, 
\begin{equation}\label{conv}
u_{i}(x,t)+ct-v_{i}(x)\to0 \ \ \textrm{uniformly on} \ \T^{n}
\end{equation}
for $i=1,2$. 
We call such a pair 
$(v_{1}(x)-ct,v_{2}(x)-ct)$ 
an \textit{asymptotic solution} of (C). 

It is worthwhile to emphasize here that for homogenization problems, 
the associated cell problems do not have the coupling terms. 
See \cite{CLL} for the detail. 
Therefore it is relatively easy to get the convergence result 
by using the classical perturbed test function method introduced by 
L. C. Evans \cite{E}. 
But when we consider the large time behavior of solutions of (C), 
we need to consider ergodic problems (E) with coupling terms. 
This fact seems to make convergence problems for large time asymptotics 
rather difficult. 
We are not yet able to justify rigorously
convergence \eqref{conv} 
for general Hamiltonians $H_i$ 
for $i=1,2$ up to now. 
We are able to handle only 
three special cases which we describe below.

\subsection{On the study of the large time behavior}
In the last decade, 
a lot of works have been devoted to the study of large time behavior of 
solutions of Hamilton--Jacobi equations 
\begin{equation}\label{eq:e-3}
u_t+H(x,Du)=0 
\quad\textrm{in} \ \Q,  
\end{equation}
where $H$ is assumed to be coercive 
and general convergence results for solutions have been established. 
More precisely, the convergence 
\begin{equation*}
u(x,t)-(v(x)-ct)\to0 \quad\textrm{uniformly on} \ x\in\T^{n} 
\ \textrm{as} \ t\to\infty 
\end{equation*}
holds, 
where $(v,c)\in C(\T^{n})\times\R$ is a solution 
of the \textit{ergodic} problem 
\begin{equation}\label{add-eigen}
H(x,Dv(x))=c \quad \textrm{in} \ \T^{n}. 
\end{equation}
Here the ergodic eigenvalue problem for $H$ is a problem of 
finding a pair of $v\in C(\T^{n})$ and $c\in\R$ 
such that $v$ is a solution of \eqref{add-eigen}. 
G. Namah and J.-M. Roquejoffre in \cite{NR} were the first to 
get general results on this convergence 
under the following additional assumptions:
$H(x,p)=F(x,p)-f(x)$,
where $F$ and $f$ satisfy 
$p\mapsto F(x,p)$ is convex for $x\in\cM$, 
\begin{gather}
F(x,p)>0 
\ \textrm{for all} \ (x,p)\in\cM\times(\R^{n} \setminus \{0\}), 
F(x,0)=0 
\ \textrm{for all} \ x \in \cM, \label{assumption-NR-1} 
\end{gather}
and
\begin{gather}
f(x) \ge 0
\ \textrm{for all} \ x \in \cM 
\ \textrm{and} \
\{f=0\} \ne \emptyset,
\label{assumption-NR-2}
\end{gather}
where $\cM$ is a smooth compact $n$-dimensional 
manifold without boundary. 
Then A. Fathi \cite{F2} proved 
the same type of convergence result 
by using general dynamical approach and weak KAM theory. 
Contrary to \cite{NR}, 
the results of \cite{F2} use 
strict convexity assumptions on $H(x,\cdot)$, 
i.e., $D_{pp}H(x,p)\ge\al I$ for all $(x,p)\in\cM\times\R^{n}$ 
and $\al>0$ (and also far more regularity) but do not need 
\eqref{assumption-NR-1}, \eqref{assumption-NR-2}. 
Afterwards J.-M. Roquejoffre \cite{R} and 
A. Davini and A. Siconolfi \cite{DS} have refined the approach 
of A. Fathi and 
they studied the asymptotic problem for 
\eqref{eq:e-3} on $\cM$ or $n$-dimensional torus. 
By another approach based on the theory of partial 
differential equations and viscosity solutions, 
this type of results has been obtained by G. Barles and 
P. E. Souganidis in \cite{BS}.
Moreover, we also refer to the literatures 
\cite{BR, I3, II1, II2, II3} 
for the asymptotic problems 
without the periodic assumptions and the periodic boundary condition  
and the literatures \cite{R, M1, M2, M3, I5, BM} 
for the asymptotic problems which treat Hamilton--Jacobi equations 
under various boundary conditions including 
three types of boundary conditions: 
state constraint boundary condition, 
Dirichlet boundary condition and 
Neumann boundary condition.  
We remark that results in \cite{BS,BR,BM} 
apply to nonconvex Hamilton--Jacobi equations. 
We refer to the literatures 
\cite{YGR, GLM1, GLM2} 
for the asymptotic problems for noncoercive Hamilton--Jacobi 
equations.

\subsection{Main results}
The first case is an analogue of the study 
by G. Namah, J.-M. Roquejoffre \cite{NR}. 
We consider Hamiltonians $H_i$ of the forms 
$$
H_i(x,p)=F_i(x,p)-f_i(x), 
$$
where the functions 
$F_{i}:\T^{n}\times\R^{n}\to[0,\infty)$ 
are coercive 
and $f_{i}:\T^{n}\to[0,\infty)$ are given continuous functions 
for $i=1,2$, respectively. 
We use the following assumptions on $F_i, f_i$.  
For $i=1,2$ 
\begin{itemize}
\item[{\rm (A2)}] 
$f_{i}(x)\ge0$ for all $x\in\T^{n}$;

\item[{\rm (A3)}] 
define
$\cA_{1}:=\{x\in\T^n\mid f_{1}(x)=0\}$,  
$\cA_{2}:=\{x\in\T^n\mid f_{2}(x)=0\}$ and  
then $\cA:=\cA_{1}\cap\cA_{2}\not=\emptyset$;  

\item[{\rm (A4)}]
there exists $\lam_{0}\in(0,1)$ such that 
\[
F_{i}(x,\lam p)\le \lam F_{i}(x,p) \ 
\textrm{for all} \ \lam\in[\lam_{0},1], \ 
x\in\T^{n}\setminus\cA 
\ \textrm{and} \ p\in\R^{n}; 
\]

%
\item[{\rm (A5)}] 
$F_{i}(x,p)\ge 0$ on $\T^{n}\times\R^{n}$,
and $F_{i}(x,0)=0$ on $\T^n$. 
\end{itemize}
With the above special forms of the Hamiltonians, 
we have 
\begin{thm}[Convergence Result 1]\label{thm:NR}
Assume that the Hamiltonians $H_i$ are of the forms
$$
H_i(x,p)=F_i(x,p)-f_i(x)
$$
and $H_i, F_i, f_i$ satisfy assumptions {\rm (A1)--(A5)}, 
then 
there exists a solution $(v_1,v_2) \in C(\T^n)^2$ of {\rm (E)} 
with $c=0$ such that convergence \eqref{conv} holds. 
\end{thm}

Notice that the {\it directional convexity condition} with respect to 
the $p$ variable on $F_i$, i.e.,
\begin{itemize}
\item[{\rm(A4')}] for any $p \in \R^n \setminus \{0\}$ 
and $x\in\T^{n}$, 
$t \mapsto F_{i}(x,tp)$ is convex, 
\end{itemize}
together with $F_i(x,0)=0$
implies (A4). 
It is clear to see that
assumption (A4) or (A4') 
does not require Hamiltonians to be convex. 
One explicit example of Hamiltonians in
Theorem \ref{thm:NR} is
$$
H_i(x,p)=F_i(x,p)-f_i(x) = 
\begin{cases}
a_{i}(x)|p|^{\alpha_i}\varphi_i(\dfrac{p}{|p|}) - f_i(x) \quad 
&\textrm{for} \ p \ne 0, \\
-f_i(x) & 
\textrm{for} \ p=0
\end{cases}
$$
for some $\alpha_i \ge 1$, $a_{i}\in C(\T^{n})$, 
$\varphi_i \in C(\mathbb{S}^{n-1})$ with $a_{i}, \varphi_{i}>0$ 
and $f_i$ satisfy (A2)--(A3) for $i=1,2$, 
where $\mathbb{S}^{n-1}$ denotes the ($n-1$)-dimensional unit sphere. 

After this work has been completed, we learned of the interesting
recent work of F. Camilli, O. Ley, P. Loreti and V. Nguyen \cite{CLLN}, 
which announces results very similar to Theorem \ref{thm:NR}.  
Their result is somewhat more general along this direction. 
In fact they consider 
systems of $m$-equations which have 
coupling terms with variable coefficients 
instead of constant coefficients. 
Also, the control-theoretic interpretation 
of (C) is derived there.

In the second case, we consider the case 
where the Hamiltonians are 
independent of the $x$ variable, i.e., 
$H_i(x,p)=H_i(p)$ for $i=1,2$.
We assume that the Hamiltonians satisfy
\begin{itemize}
\item[(A6)] $H_i$ are uniformly convex, 
i.e., 
\[
H_{i}(p)\ge H_{i}(q)+DH_{i}(q)\cdot(p-q)+\al|p-q|^{2}
\]
for some $\al>0$ and almost every $p,q\in\R^{n}$, 
\item[(A7)] $H_i(0)=0$
\end{itemize}
for $i=1,2$. 
Our main result is
\begin{thm}[Convergence Result 2]
\label{case3-thm}
Assume that $H_i(x,p)=H_i(p)$ for $i=1,2$ and
$H_i$ satisfy assumptions {\rm (A1), (A6) and (A7)},
then 
there exists a constant $M$ such that
$$
u_i(x,t)-M \to 0 \quad
\mbox{uniformly on } \T^n \ \mbox{for } i=1,2
$$
as $t \to \infty$.
\end{thm}
One explicit example of Hamiltonians in
Theorem \ref{case3-thm} is
$$
H_i(p)=|p-b_i|^2 - |b_i|^2
$$
for some constant vectors $b_i \in \R^n$
for $i=1,2$.
Notice that the above Hamiltonians
in general
do not satisfy the conditions in the first case,
particularly (A5). 
The idea for the proof of Theorem \ref{case3-thm} 
can be applied to the study 
more general forms of Hamiltonians, e.g., 
\[
H_i(x,p)=|p-\mathbf{b}_i(x)|^2 - |\mathbf{b}_i(x)|^2
\]
for $\mathbf{b}_i \in C^1(\T^n)$ 
with $\mbox{div}\, \mathbf{b}_i=0$ on $\T^n$ 
for $i=1,2$ as will be noted in Remark \ref{case3-rmk}.

In the third case, we generalize the result of 
G. Barles, P. E. Souganidis \cite{BS} for single equations to systems.
We consider the case where the two Hamiltonians $H_1, H_2$ are same, 
i.e., $H:=H_1=H_2$. 
We normalize the ergodic constant $c$ to be $0$
by replacing $H$ by $H-c$ and then 
we assume that $H$ satisfies 
\begin{itemize}
\item[{\rm(A8)}]
either of the following assumption 
{\rm (A8)$^{+}$} or {\rm (A8)$^{-}$} holds: 

\item[{\rm(A8)}$^{+}$] 
there exists $\eta_{0}>0$ such that, 
for any $\eta\in(0,\eta_{0}]$, 
there exists $\psi_{\eta}>0$ such that 
if $H(x,p+q)\ge\eta$ and $H(x,q)\le0$ for some 
$x\in\T^{n}$ and $p,q\in\R^{n}$, then for any $\mu\in(0,1]$, 
\[
\mu H(x,\frac{p}{\mu}+q)\ge 
H(x,p+q)+\psi_{\eta}(1-\mu), 
\]

\item[{\rm(A8)}$^{-}$] 
there exists $\eta_{0}>0$ such that, 
for any $\eta\in(0,\eta_{0}]$,
there exists $\psi_{\eta}>0$ such that 
if $H(x,p+q)\le-\eta$ and $H(x,q)\ge0$ for some 
$x\in\T^n$ and $p,q\in\R^{n}$, then for any $\mu\ge1$, 
\[
\mu H(x,\frac{p}{\mu}+q)\le 
H(x,p+q)-\frac{\psi_{\eta}(\mu-1)}{\mu}. 
\]
\end{itemize}
Assumption (A8)$^{+}$ was first introduced in \cite{BS} 
to replace the convexity assumption, and 
 it mainly concerns the set $\{H \geq 0\}$ 
and the behavior of $H$ in this set. 
Assumption (A8)$^{-}$ is a modified version of (A8)$^{+}$ which 
was introduced in \cite{BM} and  
on the contrary, it concerns the set $\{H \leq 0\}$. 
We can generalize them as in \cite{BS} 
but to simplify our arguments we only use the simplified version. 
See the end of Section 5.

Our third main result is 
\begin{thm}[Convergence Result 3]\label{thm:case2}
If we assume that $H=H_1=H_2$ and 
 $H$ satisfies {\rm (A1)}, {\rm (A8)} and 
the  ergodic constant $c$ is equal to $0$, 
then there exist a solution 
$(v,v) \in C(\T^n)^2$  of {\rm (E)} with 
$c=0$ such that convergence \eqref{conv} holds. 
\end{thm}
We notice that if $H$ is smooth with respect to the $p$-variable, 
then (A8) is equivalent
to a \textit{one-sided directionally strict convexity} 
in a neighborhood of $\{p\in\R^n\mid H(x,p)=0\}$ 
for all $x\in\T^n$, i.e., 
\begin{itemize}
\item[{\rm(A8')}]
either of the following assumption 
{\rm (A8')$^{+}$} or {\rm (A8')$^{-}$} holds: 

\item[{\rm(A8')}$^{+}$] 
there exists $\eta_{0}>0$ such that, 
for any $\eta\in(0,\eta_{0}]$, 
there exists $\psi_{\eta}>0$ such that 
if $H(x,p+q)\ge\eta$ and $H(x,q)\le0$ for some 
$x\in\T^{n}$ and $p,q\in\R^{n}$, then for any $\mu\in(0,1]$, 
\[
D_{p}H(x,p+q)\cdot p-H(x,p+q)\ge\psi_{\eta}, 
\]

\item[{\rm(A8')}$^{-}$] 
there exists $\eta_{0}>0$ such that, 
for any $\eta\in(0,\eta_{0}]$,
there exists $\psi_{\eta}>0$ such that 
if $H(x,p+q)\le-\eta$ and $H(x,q)\ge0$ for some 
$x\in\T^n$ and $p,q\in\R^{n}$, then for any $\mu\ge1$, 
\[
D_{p}H(x,p+q)\cdot p-H(x,p+q)\ge\psi_{\eta}. 
\]
\end{itemize}
We refer the readers to \cite{BS}
for interesting examples of 
Hamiltonians in Theorem \ref{thm:case2}. 
Our conclusions in Cases 2, 3 seem to go beyond the recent work 
\cite{CLLN}.

This paper is organized as follows: 
in Section 2 we give some preliminary results. 
Section 3, Section 4, and Section 5
 are respectively devoted 
to the proofs of 
Theorems \ref{thm:NR}--\ref{thm:case2}.
In Appendix we present the proof of the result 
on ergodic problems.

\medskip
\noindent
{\bf Notations.} 
For $A\subset\R^{n}$ and $k\in\N$, 
we denote by 
$C(A)$, $\LSC(A)$, $\USC(A)$ and $C^k(A)$ 
the space of real-valued 
continuous, 
lower semicontinuous, 
upper semicontinuous  
and $k$-th continuous differentiable functions  
on $A$, respectively. 
We denote by $\W(A)$ the set of bounded functions 
whose first weak derivatives are essentially bounded. 
We call a function $m:[0,\infty)\to[0,\infty)$ 
a modulus if it is continuous and nondecreasing 
on $[0,\infty)$ and vanishes at the origin.


\section{Preliminaries}
In this section we assume only (A1). 
\begin{prop}[Ergodic Problem (E)
{{\rm (e.g.,} \cite[Theorem 4.2]{CGT2}{\rm )}}]\label{prop:cell}
There exists 
$(v_{1},v_{2},\ol{H}_1, \ol{H}_2)\in\W(\T^{n})^{2}\times\R^{2}$
of 
\begin{equation}\label{cell}
\begin{cases}
H_1(x, Dv_{1})+c_1(v_{1}-v_{2})=\ol{H}_{1} \ 
\textrm{in} \ \T^{n}, \\
H_2(x, Dv_{1})+c_2(v_{2}-v_{1})=\ol{H}_{2} \ 
\textrm{in} \ \T^{n}. 
\end{cases}
\end{equation}
Furthermore, $c_2 \ol{H}_{1} +c_1 \ol{H}_{2}$ is unique.
\end{prop}
We note that solutions $v_1, v_2$ of \eqref{cell} 
are not unique in general even up to constants.
Also it is easy to see that $\ol{H}_1, \ol{H}_2$
are not unique as well.
Take $v_1'=v_1+C_1,v_2'=v_1+C_2$
for some constants $C_1,C_2$
then 
$$\ol{H}_1'=\ol{}H_1+c_1(C_1-C_2),\,
\ol{H}_2'=\ol{}H_1+c_2(C_2-C_1),
$$
which shows that $\ol{H}_i$ can individually
take any real value. 
But remarkably, we have 
$$
c_2 \ol{H}_{1} +c_1 \ol{H}_{2}=c_2 \ol{H}_{1}' +c_1 \ol{H}_{2}', 
$$ 
which is a unique constant. 
We can get the existence result by an argument 
similar to a classical result in \cite{LPV} 
(see also the proof of Proposition \ref{prop:case1:erg} below). 
We give the sketch of the proof for the uniqueness of 
$c_2 \ol{H}_{1} +c_1 \ol{H}_{2}$ 
in Appendix for the reader's convenience.

We assume henceforth for simplicity that $c_1=c_2=1$. 
Then the ergodic constant $c$ is unique and is given by 
$$
c=\dfrac{\ol{H}_1+\ol{H}_2}{2}.
$$
The comparison principle for (C) is a classical result. 
See \cite{LEN88, EL, IK1}, \cite[Proposition 3.1]{CLL} for instance. 
\begin{prop}[Comparison Principle for (C)]\label{prop:comp-for-C}
Let $(u_{1}, u_{2})\in\USC(\cQ)^{2}$, $(v_{1}, v_{2})\in \LSC(\cQ)^{2}$ 
be a subsolution and a supersolution of {\rm (C)}, 
respectively. 
If $u_{i}(\cdot,0)\le v_{i}(\cdot,0)$ on $\T^{n}$, 
then $u_{i}\le v_{i}$ on $\cQ$ for $i=1,2$. 
\end{prop}

The following proposition is a straightforward 
application of Propositions \ref{prop:cell}, \ref{prop:comp-for-C}.

\begin{prop}[Boundedness of Solutions of (C)]\label{prop:bdd}
Let $(u_{1}, u_{2})$ be the solution of {\rm (C)} and 
let $c$ be the ergodic constant for {\rm (E)}. 
Then we have $|u_{i}(x,t)+ct|\le C$ on $\cQ$ for some $C>0$ 
for $i=1,2$. 
\end{prop}

In view of the coercivity assumption on $H_{i}$ for $i=1,2$, 
we have the following Lipschitz regularity result. 
\begin{prop}[Lipschitz Regularity of Solutions of (C)]\label{prop:lip}
If $u_{0i}\in\W(\T^{n})$ for $i=1,2$, 
then $(u_{1}+ct, u_{2}+ct)$ 
is in $\W(\cQ)^{2}$, 
where $(u_{1}, u_{2})$ is the solution of {\rm (C)} and 
$c$ is the ergodic constant.  
\end{prop}
We \textit{assume} henceforth that $u_{0i}\in\W(\T^{n})$ 
for $i=1,2$
in order to avoid technicalities but they are not necessary. 
We can easily remove these additional requirements on $u_{0i}$. 
See Remark \ref{rem:regularity} for details.


\section{First Case}
In this section we consider the case where Hamiltonians
have the forms $H_i(x,p)=F_i(x,p)-f_i(x)$,
and $H_i, F_i, f_i$ satisfy assumptions (A1)--(A5).
System (C) becomes
\begin{numcases}
{\textrm{(C1)} \hspace{.3cm}}
(u_{1})_t + F_{1}(x,Du_{1}) +u_{1}-u_{2} = f_{1}(x)
& in $\Q$, \nonumber \\
(u_{2})_t + F_{2}(x,Du_{2}) + u_{2}-u_{1} = f_{2}(x)
& in $\Q$, \nonumber\\
u_{1}(x,0)=u_{01}(x),
 \ 
u_{2}(x,0)=u_{02}(x)
& 
on $\T^{n}$. \nonumber
\end{numcases}
In order to prove Theorem \ref{thm:NR}, we need several following steps.

\subsection{Stationary Problems}

\begin{prop}\label{prop:case1:erg}
The ergodic constant $c$ is equal to $0$. 
\end{prop}
\begin{proof}
For $\ep>0$ let us consider 
a usual approximate monotone system
\begin{equation}\label{pf:additive1}
\left\{
\begin{aligned}
&
F_{1}(x,Dv_{1}^{\ep}(x))+(1+\ep)v_{1}^{\ep}-v_{2}^{\ep}=f_{1}(x) 
&& \textrm{in} \ \T^{n}, \\
&
F_{2}(x,Dv_{2}^{\ep}(x))+(1+\ep)v_{2}^{\ep}-v_{1}^{\ep}=f_{2}(x) 
&& \textrm{in} \ \T^{n}. 
\end{aligned}
\right.
\end{equation}
It is easy to see that $(0, 0)$, 
$(C_{1}/\ep, C_{1}/\ep)$ are a subsolution and a supersolution 
of the above for $C_{1}>0$ large enough. 
By Perron's method for the monotone system, 
we have a unique solution $(v_{1}^{\ep}, v_{2}^{\ep})\in C(\T^{n})^2$
of \eqref{pf:additive1}.  
By the way of construction we have 
\begin{equation}\label{pf:additive2} 
0\le\ep v_{i}^{\ep}\le C_{1} \ \textrm{on} \ \T^{n} 
\end{equation}
for $i=1,2$.  
Summing up both equations in \eqref{pf:additive1}, 
we have 
\[
F_{1}(x,Dv_{1}^{\ep})+F_{2}(x,Dv_{2}^{\ep})
=-\ep (v_{1}^{\ep}+v_{2}^{\ep})+
f_{1}(x)+f_{2}(x)\le C_{2} 
\]
for some $C_{2}>0$. 
By the coercivity of $F_{i}$ we obtain 
\[
\|Dv_{i}^{\ep}\|_{L^\infty(\T^{n})} \le C_{2}
\]
for $i=1,2$ by replacing $C_{2}$ by a larger constant 
if necessary. 
Therefore we see that $\{v_{i}^{\ep}\}_{\ep \in (0,1)}$ 
are equi-Lipschitz continuous.

We claim that there exists a constant $C_{3}>0$ 
\begin{equation}\label{pf:additive3}
|v_{1}^{\ep}(x)-v_{2}^{\ep}(y)|\le C_{3} \ 
\textrm{for all} \ x,y\in\T^{n}.
\end{equation}
Indeed setting 
$m_{i}^{\ep}:=\max_{\T^{n}}v_{i}^{\ep}=v_{i}^{\ep}(z_{i})$ 
for some $z_{i}\in\T^{n}$ for $i=1,2$.
Take $0$ as a test function in 
the first equation of \eqref{pf:additive1}
to derive
\[
F_{1}(z_{1},0)+(1+\ep)v_{1}^{\ep}(z_{1})-v_{2}^{\ep}(z_{1})
\le f_{1}(z_{1}), 
\]
which implies 
\[
v_{1}^{\ep}(z_{1})-v_{2}^{\ep}(z_{1})
\le 
-F_{1}(z_{1},0)-\ep v_{1}^{\ep}(z_{1})+f_{1}(z_{1})
\le C_{3}
\]
for some $C_{3}>0$. 
Thus,  
\begin{align*}
v_{1}^{\ep}(x)-v_{2}^{\ep}(y)
\le&\, 
v_{1}^{\ep}(z_{1})-v_{2}^{\ep}(y)\\
{}=&\, 
v_{1}^{\ep}(z_{1})-v_{2}^{\ep}(z_{1})
+v_{2}^{\ep}(z_{1})-v_{2}^{\ep}(y)
\le C_{3}
\end{align*}
by replacing $C_{3}$ by a larger constant 
if necessary. 
This implies \eqref{pf:additive3}.
In particular, 
$|m_{1}^{\ep}-m_{2}^{\ep}|\le C_{3}$.

Let $w_{i}^{\ep}(x):=v_{i}^{\ep}(x)-m_{i}^{\ep}$. 
Because of \eqref{pf:additive2}, 
$\{w_{i}^{\ep}\}_{\ep\in(0,1)}$ is a sequence of equi-Lipschitz continuous 
and uniformly bounded functions on $\T^{n}$. 
Moreover they satisfy  
\begin{equation*}
\left\{
\begin{aligned}
&
F_{1}(x,Dw_{1}^{\ep}(x))+(1+\ep)w_{1}^{\ep}-w_{2}^{\ep}=f_{1}(x)
-(1+\ep)m_{1}^{\ep}+m_{2}^{\ep}
&& \textrm{in} \ \T^{n}, \\
&
F_{2}(x,Dw_{2}^{\ep}(x))+(1+\ep)w_{2}^{\ep}-w_{1}^{\ep}=f_{2}(x)
-(1+\ep)m_{2}^{\ep}+m_{1}^{\ep}
&& \textrm{in} \ \T^{n} 
\end{aligned}
\right.
\end{equation*}
in the viscosity solution sense. 
By Ascoli-Arzela's theorem, there exists a sequence $\ep_j \to 0$ so that
\begin{gather*}
w_{i}^{\ep_j}\to w_{i}, \\
-(1+\ep_j)m_{1}^{\ep_j}+m_{2}^{\ep_j}\to \ol{H}_1 \ 
\textrm{and} \ 
-(1+\ep_j)m_{2}^{\ep_j}+m_{1}^{\ep_j}\to \ol{H}_2 \ 
\end{gather*}
uniformly on $\T^{n}$ 
as $j\to\infty$ 
for some $(w_{1},w_{2})\in W^{1,\infty}(\T^{n})^{2}$ and 
$(\ol{H}_1,\ol{H}_2)\in\R^2$. 
By a standard stability result of viscosity solutions 
we see that $(w_{1},w_{2},\ol{H}_1, \ol{H}_2)$ is a solution of \eqref{cell}.

We now prove that $c:=(\ol{H}_{1}+ \ol{H}_{2})/2=0$. 
Noting that 
$m_{i}^{\ep_j}\ge0$ and 
\[
\frac{1}{2}
\bigl\{
(-(1+\ep_j)m_{1}^{\ep_j}+m_{2}^{\ep_j})
+
(-(1+\ep_j)m_{2}^{\ep_j}+m_{1}^{\ep_j})
\bigr\}
=
-\frac{1}{2}\ep_j(m_{1}^{\ep_j}+m_{2}^{\ep_j})
\to c 
\]
as $j\to\infty$, we see that $c\le0$. 
Furthermore, summing up the two equations in \eqref{cell}, 
we obtain 
\begin{align*}
2c=\ol{H}_1+\ol{H}_2
=F_{1}(x,Dw_{1})+F_{2}(x,Dw_{2})-f_{1}(x)-f_{2}(x)
\ge-f_{1}(x)-f_{2}(x)
\end{align*}
for almost every $x\in\T^{n}$. 
Since $\cA\not=\emptyset$, we see that $c\ge0$.  
Together with the above observation we get the conclusion. 
\end{proof}

\begin{thm}[Comparison Principle for Stationary Problems]\label{thm:comp-for-E}
Let $(u_{1}, u_{2})\in\USC(\T^{n})^{2}$, $(v_{1}, v_{2})\in\LSC(\T^{n})^{2}$ 
be, respectively, a subsolution and a supersolution of 
\begin{numcases}
{{\rm (S1)} \hspace{1cm}}
F_{1}(x,Dv_{1}(x))+v_{1}-v_{2}=f_{1}(x) 
& in  $\T^{n}$, \nonumber\\
F_{2}(x,Dv_{2}(x))+v_{2}-v_{1}=f_{2}(x) 
& in  $\T^{n}$. \nonumber
\end{numcases}
If $u_{i}\le v_{i}$ on $\cA$, then 
$u_{i}\le v_{i}$ on $\T^{n}$ 
for $i=1,2$.  
\end{thm}

The idea of the proof below basically comes from 
the combination of those in \cite{Ish87} and 
\cite{LEN88,EL,IK1}. 
It is worthwhile to mention that 
the set $\cA$ plays the role of the boundary as in 
\cite{FS,IM}. 
See also \cite{CL} and \cite[Theorem 3.3]{CLLN}
 for weakly coupled systems of Hamilton--Jacobi equations.

\begin{proof}
Fix any $\del>0$. 
We may choose an open neighborhood $V$ of $\cA$ and 
$\ol{\lam}\in[\lam_{0},1)$ so that 
$\lam u_{i}\le v_{i}+\del$ on $V$ for $\lam\in[\ol{\lam},1]$ and 
$i=1,2$, where $\lam_{0}$ is the constant in (A2). 
It is enough to show that 
$\lam u_{i}\le v_{i}+\del$ on $\T^{n}\setminus V$ 
for $\lam\in[\ol{\lam},1]$. 
Fix $\lam\in [\ol{\lam},1]$ and 
we set $u_{i}^{\lam}:=\lam u_{i}$ and 
$v_{i}^{\del}:=v_{i}+\del$. 
We prove 
the above statement 
 by a contradiction argument. 
Suppose that 
$M:=\max_{i=1,2, x\in\T^{n}\setminus V}(u_{i}^{\lam}-v_{i}^{\del})(x)>0$.

We take $i_{0}\in\{1,2\}$, $\xi\in\T^{n}\setminus V$ such that 
$M=(u_{i_{0}}^{\lam}-v_{i_{0}}^{\del})(\xi)$. 
We may assume that $i_{0}=1$ by symmetry.  
We first consider the case where 
\begin{equation}\label{pf:case1}
M_{\lam}
=(u_{1}^{\lam}-v_{1}^{\del})(\xi)
=(u_{2}^{\lam}-v_{2}^{\del})(\xi). 
\end{equation}

We define the function $\Psi:\T^{2n}\to\R$ by 
\[
\Psi(x,y):=
u_{1}^{\lam}(x)-v_{1}^{\del}(y)
-\frac{|x-y|^{2}}{2\ep^{2}}-\frac{|x-\xi|^{2}}{2}. 
\]
Let $\Psi$ achieve its maximum at some point 
$(x_{\ep}, y_{\ep})\in\T^{2n}$. 
By the definition of viscosity solutions we have 
\begin{align*}
& 
F_{1}(x_{\ep}, \frac{1}{\lam}\bigl(\frac{x_{\ep}-y_{\ep}}{\ep^{2}}
+x_{\ep}-\xi\bigr))+(u_{1}-u_{2})(x_{\ep})\le f_{1}(x_{\ep}), \\
& 
F_{1}(x_{\ep}, \frac{x_{\ep}-y_{\ep}}{\ep^{2}})
+(v_{1}-v_{2})(y_{\ep})\ge f_{1}(y_{\ep}). 
\end{align*}

By the usual argument we may assume that  
\[
x_{\ep}, y_{\ep}\to\xi, \ 
\frac{x_{\ep}-y_{\ep}}{\ep^{2}}\to p\in\R^{n}
\]
as $\ep\to0$ by taking a subsequence if necessary 
in view of the Lipschitz continuity of solutions.  
Therefore sending $\ep$ to $0$ yields 
\begin{align}
& 
F_{1}(\xi, \frac{p}{\lam})+(u_{1}-u_{2})(\xi)\le f_{1}(\xi), 
\label{pf:comp-1}\\
& 
F_{1}(\xi, p)+(v_{1}-v_{2})(\xi)\ge f_{1}(\xi). \label{pf:comp-2}
\end{align}
In view of (A4), 
\eqref{pf:comp-1} transforms to read 
\begin{equation}\label{pf:comp-3}
F_{1}(\xi, p)+(u_{1}^{\lam}-u_{2}^{\lam})(\xi)
\le \lam f_{1}(\xi) 
\ \textrm{for all} \ \lam\in[\ol{\lam},1] .
\end{equation}
Note that $(v_{1}-v_{2})(\xi)=(v_{1}^{\del}-v_{2}^{\del})(\xi)$. 
By \eqref{pf:case1}, \eqref{pf:comp-2} and \eqref{pf:comp-3} 
we get $f_{1}(\xi)\le\lam f_{1}(\xi)$.
Similarly, $f_{2}(\xi)\le\lam f_{2}(\xi)$.
Hence $f_{1}(\xi)+f_2(\xi)\le\lam(f_{1}(\xi)+f_2(\xi))$
which is a contradiction since $f_{1}(\xi)+f_2(\xi)>0$ and $\lam\in(0,1)$.

We next consider the case where 
\[
(u_{1}^{\lam}-v_{1}^{\del})(\xi)
\not=(u_{2}^{\lam}-v_{2}^{\del})(\xi).
\]
Then there exists $a>0$ such that 
$(u_{1}^{\lam}-v_{1}^{\del})(\xi)\ge 
(u_{2}^{\lam}-v_{2}^{\del})(\xi)+a$ and therefore 
by \eqref{pf:comp-2}, \eqref{pf:comp-3} 
we obtain 
\[
0>(\lam-1)f_{1}(\xi)\ge 
(u_{1}^{\lam}-v_{1}^{\del})(\xi)-(u_{2}^{\lam}-v_{2}^{\del})(\xi)
\ge a,  
\]
which is a contradiction. 
This finishes the proof. 
\end{proof}

\subsection{Convergence}
\begin{prop}[Monotonicity Property $1$]\label{prop:mon1}
Set 
$U(x,t):=u_{1}(x,t)+u_{2}(x,t)$. 
Then the function $t\mapsto U(x,t)$ 
is nonincreasing for all $x\in\cA$. 
\end{prop}
\begin{proof}
It is easy to see that $U$ satisfies 
$U_{t}\le0$ on $\cA$ in the viscosity sense 
and we get the conclusion. 
\end{proof}

\begin{prop}[Monotonicity Property $2$]\label{prop:mon2}
Set 
\[
V(x,t)
:=\max\{u_{1}(x,t),u_{2}(x,t)\}
=\frac{1}{2}\bigl\{(u_{1}+u_{2})(x,t)+
|(u_{1}-u_{2})(x,t)|\bigr\}. 
\]
Then the function $t\mapsto V(x,t)$ 
is nonincreasing for all $x\in\cA$. 
\end{prop}

We notice that the result of Proposition \ref{prop:mon2} is 
included by the recent result of \cite[Remark 5.7, (3)]{CLLN}  
but our proof seems to be more direct.

\begin{proof}
Fix $x\in\cA$. 
For $\ep, \del>0$ we set 
$K_{\ep}(x):=x+[-\ep,\ep]^{n}$ and 
\[
V_{\del}(x,t):=
\frac{1}{2}\bigl((u_{1}+u_{2})(x,t)+
\langle (u_{1}-u_{2})(x,t)\rangle_{\del}\bigr), 
\]
where $\langle p\rangle_{\del}:=\sqrt{|p|^{2}+\del^{2}}$. 
We note that 
$V_{\del}$ converges uniformly  to $V$ as $\del \to 0$.

We have for all $t, h\ge0$ 
\[
\int_{K_{\ep}(x)} 
V_{\del}(y,t+h)-V_{\del}(y,t)\, dy 
=
\int_{K_{\ep}(x)\times[t,t+h]}(V_{\del})_{t}(y,s)\,dy\,ds. 
\]
Let $(y,s)$ be a point at which $u_{1}, u_{2}$ are differentiable. 
We calculate that 
\begin{align*}
&(V_{\del})_{t}(y,s)\\
=&\, 
\frac{1}{2}\bigl\{
(u_{1})_{t}+(u_{2})_{t}+\frac{u_{1}-u_{2}}{\langle u_{1}-u_{2}\rangle_{\del}}
((u_{1})_{t}-(u_{2})_{t})
\bigr\}\\
=&\, 
\frac{1}{2}\bigl\{
f_{1}+f_{2}+
\frac{u_{1}-u_{2}}{\langle u_{1}-u_{2}\rangle_{\del}}
(f_1-f_2)\bigr\}\\
&\, 
+\frac{1}{2}\bigl\{
-F_{1}-F_{2}+
\frac{u_{1}-u_{2}}{\langle u_{1}-u_{2}\rangle_{\del}}
(F_2-F_1)\bigr\}
-\frac{1}{\langle u_{1}-u_{2}\rangle_{\del}}(u_{1}-u_{2})^{2}\\
\le&\, 
\frac{1}{2}\bigl\{
f_{1}+f_{2}+
\frac{u_{1}-u_{2}}{\langle u_{1}-u_{2}\rangle_{\del}}
(f_1-f_2)\bigr\}
+\frac{1}{2}\bigl\{
-F_{1}-F_{2}+
\frac{u_{1}-u_{2}}{\langle u_{1}-u_{2}\rangle_{\del}}
(F_2-F_1)\bigr\}.
\end{align*}

In view of (A5) and (A3) sending $\del\to0$ yields 
\begin{align*}
&
\int_{K_{\ep}(x)} 
V(y,t+h)-V(y,t)\, dy \\
\le& 
\int_{K_{\ep}(x)\times[t,t+h]}
\frac{1}{2}\bigl\{
f_{1}+f_{2}+
\sgn(u_{1}-u_{2})(f_1-f_2)\bigr\}\\
&\,+
\frac{1}{2}\bigl\{
-F_{1}-F_{2}+
\sgn(u_{1}-u_{2})(F_2-F_1)\bigr\}\, dy ds\\
\le& 
\int_{K_{\ep}(x)\times[t,t+h]}
\frac{1}{2}\bigl\{
f_{1}+f_{2}+
\sgn(u_{1}-u_{2})(f_1-f_2)\bigr\}\,dy ds\\
\le& 
\int_{K_{\ep}(x)\times[t,t+h]}
\om_{f_{1}}(|x-y|)+\om_{f_{2}}(|x-y|)\,dy ds\\
\le& 
\,\ep^{n}h(\om_{f_{1}}(\sqrt{n}\ep)+\om_{f_{2}}(\sqrt{n}\ep)), 
\end{align*}
where $\om_{f_{i}}$ are the moduli of continuity of $f_{i}$ 
for $i=1,2$. 
By dividing by $\ep^{n}$ and sending $\ep\to0$ 
we get the conclusion. 
\end{proof}

\begin{proof}[Proof of Theorem {\rm \ref{thm:NR}}]
For any $x\in\cA$ by Propositions \ref{prop:mon1}, \ref{prop:mon2} 
we see that 
$(u_{1}+u_{2})(x,t)\to\al(x)$ and 
$|(u_{1}-u_{2})(x,t)|\to\beta(x)$ 
as $t\to\infty$. 
If $\beta(x)>0$, then $(u_{1}-u_{2})(x,t)$ converges
as $t\to\infty$ since $t \mapsto (u_{1}-u_{2})(x,t)$ is continuous.
The limit may be either $\beta(x)$ or $-\beta(x)$. 
Therefore $u_{1}(x,t), u_{2}(x,t)$ converge as $t\to\infty$. 
If $\beta(x)=0$, then we have 
\[
(u_{1}+u_{2})(x,t)-|(u_{1}-u_{2})(x,t)|
\le 2 u_{1}(x,t)
\le (u_{1}+u_{2})(x,t)+|(u_{1}-u_{2})(x,t)|, 
\]
which implies $u_{1}(x,t)$ and $u_2(x,t)$ 
converge to $(1/2)\al(x)$ as $t\to\infty$. 
Consequently, we see that $u_{1}(x,t)$, $u_2(x,t)$ 
converge for all $x\in\cA$ as $t\to\infty$.

Now, let us define the following half-relaxed semilimits
$$
\ol{u}_i(x)={\limsup_{t\to\infty}}^{*} [u_i](x,t) 
\ \textrm{and} \ 
\, \ul{u}_i(x)=\liminf_{t\to\infty}\mbox{}_{*}[u_i](x,t) 
$$
for $x \in \T^n$ and $i=1,2$.
By standard stability results of the theory of viscosity solutions,
$(\ol{u}_1,\ol{u}_2)$, $(\ul{u}_1,\ul{u}_2)$ are a subsolution
and a supersolution of (E), respectively. 
Moreover, $(\ol{u}_1,\ol{u}_2)=(\ul{u}_1,\ul{u}_2)$ on $\cA$, 
since $u_{1}, u_{2}$ converge on $\cA$ as $t\to\infty$.
By the comparison principle, Theorem \ref{thm:comp-for-E}, 
we obtain $(\ol{u}_1,\ol{u}_2)=(\ul{u}_1,\ul{u}_2)$ in $\T^n$ 
and the proof is complete.
\end{proof}

\begin{rem}\label{rem:regularity}
(i) 
The Lipschitz regularity assumption 
on $u_{0i}$ for $i=1,2$ is convenient to 
avoid technicalities but it is not necessary. 
We can remove it as follows. 
For each $i$,
we may choose a sequence 
$\{u_{0i}^{k}\}_{k\in\N}\subset \W(\T^{n})$ 
so that $\|u_{0i}^{k}-u_{0i}\|_{\Li(\T^{n})} \le 1/k$ 
for all $k\in\N$. 
By the maximum principle, we have 
\[
\|u_{i}-u^{k}_{i}\|_{\Li(\Q)}
\le 
\|u_{0i}-u_{0i}^{k}\|_{\Li(\T^n)} \le 1/k, 
\]
and therefore 
\[
u^{k}_{i}(x,t)-1/k\le u_{i}(x,t)\le u^{k}_{i}(x,t)+1/k 
\ \textrm{for all} \ (x,t)\in\cQ, 
\]
where $(u_1,u_2)$ is the solution of (C) 
and $(u^{k}_{1},u^{k}_{2})$ are the solutions 
of (C) with $u_{0i}=u_{0i}^{k}$ for $i=1,2$. 
Therefore we have 
\[
u_{\infty i}^{k}(x)-1/k
\le 
\limiinf_{t\to\infty}u_{i}(x,t)
\le 
\limssup_{t\to\infty}u_{i}(x,t)
\le 
u_{\infty i}^{k}(x)+1/k
\]
for all $x\in\T^{n}$, 
where $u_{\infty i}^{k}(x):=\lim_{t\to\infty}u^{k}_{i}(x,t)$. 
This implies that 
\[
\limiinf_{t\to\infty}u_{i}(x,t)
=
\limssup_{t\to\infty}u_{i}(x,t)
\]
for all $x\in\T^n$ and $i=1,2$. \\
(ii) 
Notice that if $\cA =\emptyset$ then 
the comparison principle for (S1) holds, 
i.e., 
for any subsolution $(v_{1},v_{2})$ and 
any supersolution $(w_{1},w_{2})$ 
we have $v_{i}\le w_{i}$ on $\T^n$ 
for $i=1,2$ (e.g., \cite[Theorem 3.3]{CL}). 
This fact implies that the ergodic constant $c$ is negative 
(not $0$!). 
Indeed, by the argument same as in the proof of 
Proposition \ref{prop:case1:erg} we easily see that 
$c\le0$. 
Suppose that $c=0$ and then the comparison principle 
implies that (E) has a unique solution $(v_1, v_2)$. 
However, that is obviously not correct since 
for any solution $(v_1, v_2)$ of (E), 
$(v_1+C, v_2+C)$ is also a solution for any 
constant $C$. 
In this case we do not know whether 
convergence \eqref{conv} holds or not. 
\end{rem}

\subsection{Systems of $m$-equations}

This section was added after we had received the draft \cite{CLLN} 
in order for the readers to see the different ideas used 
in our work and \cite{CLLN}.

In this subsection we consider weakly coupled systems 
of $m$-equations for $m \ge 2$ 
\[
(u_i)_t+F_{i}(x,Du_i)+\sum_{j=1}^{m}c_{ij}u_j=f_i 
\ \textrm{in} \ \Q 
\ \textrm{for} \ i=1,\ldots,m, 
\]
where $F_i$ satisfy (A1), (A5) and the convexity with respect to 
the $p$-variable, 
\begin{equation}\label{cond:cij}
c_{ii}\ge 0, \ 
c_{ij}\le0 \ \textrm{if} \ i\not=j \ 
\textrm{and} \ 
\sum_{i=1}^{m}c_{ij}=\sum_{j=1}^{m}c_{ij}=0
\end{equation}
for $i,j\in\{1,\ldots,m\}$ and $f_i$ satisfy (A2) and 
\[
\cA:=\bigcap_{i=1}^{m}\{x\in\T^n\mid f_i(x)=0\}\not=\emptyset
\]
then the result of Theorem \ref{thm:NR} still holds.
In \cite{CLLN} the authors first found the importance of 
irreducibility of coupling term. 
Although it is not essential in our argument, 
we also somehow use it below.  
Let us first assume for simplicity  
that the coefficient matrix $(c_{ij})$ is irreducible, i.e., 
\begin{itemize}
\item[(M)] For any $I \varsubsetneq \{1,\ldots,m\}$,
there exist $i \in I$ and $j \in \{1,\ldots,m\} \setminus I$
such that $c_{ij} \neq 0$.
\end{itemize}
Condition (M) will be removed in Remark \ref{rem:irr} 
at the end of this subsection.

We just give a sketch of the formal proof for the convergence. 
By a standard regularization argument we can prove it rigorously 
in the viscosity solution sense.

We only need to prove the convergence of $u_i$ on $\cA$ 
for each $i\in\{1,\ldots,m\}$, 
since we have an analogous comparison principle 
to Theorem \ref{thm:comp-for-E} 
when (M) holds.
For $(x,t)\in\cQ$, we can choose $\{i_{x,t}\}_{i=1}^m$
such that $\{1_{x,t},\ldots,m_{x,t}\}=\{1,\ldots,m\}$ and 
\[
u_{1_{x,t}}(x,t)\ge u_{2_{x,t}}(x,t)\ge\ldots\ge 
u_{m_{x,t}}(x,t)
\]
and set $v_i(x,t):=u_{i_{x,t}}(x,t)$. 

Fix $(x_{0},t_{0})\in\cA\times(0,\infty)$ and we may assume 
without loss of generality that 
\[
1_{x_{0},t_{0}}=1 \ \textrm{and} \ 
2_{x_{0},t_{0}}=2. 
\]
Noting that $c_{1j}\le0$, $u_1\ge u_j$ 
for all $j=2,\ldots,m$, and $F_1\ge0$, 
we have  
\[
(v_1)_t
=
(u_1)_t
\le
(u_1)_t+\sum_{j=1}^{m}c_{1j}u_1
\le 
(u_1)_t+F_{1}(x_{0},Du_1)+\sum_{j=1}^{m}c_{1j}u_j
=0  
\]
at the point $(x_{0},t_{0})$, 
which implies that 
$v_1(x_0,\cdot)$ is nonincreasing for $x_{0}\in\cA$
and therefore $v_{1}(x_0,\cdot)$ converges as $t\to\infty$.

Noting that $u_2\ge u_j$ and $c_{ij}\le0$ 
for all $i=1,2$, $j=3,\ldots,m$, $\sum_{j=1}^{m}c_{2j}=0$, 
and $F_i\ge0$, 
we have 
\begin{align*}
(v_1+v_2)_t &=(u_1+u_2)_t \le (u_1+u_2)_t +
\sum_{i=1}^2 \sum_{j=3}^m c_{ij} (u_j-u_2)\\
&= (u_1)_t+(u_2)_t +(c_{11}+c_{12}+c_{21}+c_{22})u_2 +
\sum_{i=1}^2 \sum_{j=3}^m c_{ij} u_j\\
&\le (u_1)_t+(u_2)_t +(c_{11}+c_{21})u_1+
(c_{12}+c_{22})u_2 +
\sum_{i=1}^2 \sum_{j=3}^m c_{ij} u_j\\
&\le (u_1)_t+(u_2)_t +F_1(x_0,Du_1)+F_2(x_0, Du_2)+
\sum_{i=1}^2 \sum_{j=1}^m c_{ij} u_j=0
\end{align*}
at the point $(x_{0},t_{0})$. 
Thus, 
\[
(v_1+v_2)_{t}(x_{0},t_{0})\le0. 
\]
Therefore $(v_1+v_2)(x_0,\cdot)$ is nonincreasing 
for $x_{0}\in\cA$.
Since we have already known that $v_{1}(x_{0},\cdot)$ 
converges, we see that $v_{2}(x_{0},\cdot)$ converges 
as $t\to\infty$.

By the induction argument, we can prove that
$(v_1+\ldots+v_k)(x_0,\cdot)$ is nonincreasing for all $x_0\in \cA$
and $k \in \{1,\ldots,m\}$, which is a geralization of 
Proposition \ref{prop:mon2}. 
Thus, we see that 
\[
v_{i}(x_{0},t)\to w_{i}(x_{0}) \ 
\textrm{as} \ t\to\infty \ 
\textrm{for} \ i\in\{1,\ldots,m\}, 
\]
which concludes that each $u_{i}(x_{0},t)$ converges as 
$t\to\infty$ for $x_0 \in \cA$.

\begin{rem}\label{rem:irr}
(i) In general, condition (M) can be removed as follows.
By possible row and column permutations, $\cC:=(c_{ij})$ can
be written in the block triangular form
$$
\cC=(\cC_{pq})_{p,q=1}^l
$$
where $\cC_{pq}$ are $s_p\times s_q$ matrices
for $p,q \in \{1,\ldots,l\}$,
$\sum_{k=1}^l s_k=m$, $\cC_{kk}$ are irreducible
for $k \in \{1,\ldots,l\}$
and $\cC_{pq}=0$ for $p >q$ as in \cite{BS1}.
By \eqref{cond:cij}, we can easily see that $\cC_{pq}=0$
for $p<q$ as well. 
Therefore the convergence result above can
be applied to each irreducible matrix $\cC_{kk}$ to yield the result.
\\
(ii) Our approach in this general case is slightly different from
the one in \cite{CLLN}. The convergence of each $u_i(x,t)$ as $t \to \infty$
for $i \in \{1,\ldots,m\}$, for $x\in \cA$ plays the key role here, while 
Lemma 5.6 plays the key role in \cite{CLLN}. 
See Lemma 5.6 in \cite{CLLN} for more details.
\end{rem}


\section{Second case}

In this section we study the case where Hamiltonians 
are independent of the $x$-variable and then 
(C) reduces to 
\begin{numcases}
{\textrm{(C2)} \hspace{1cm}}
(u_{1})_t + H_{1}(Du_{1}) + u_{1}-u_{2} = 0
& in $\Q$, \label{eq-c3-1} \\
(u_{2})_t + H_{2}(Du_{2}) + u_{2}-u_{1} = 0
& in $\Q$, \label{eq-c3-2} \\
u_{1}(x,0)=u_{01}(x), \ 
u_{2}(x,0)=u_{02}(x) 
& 
on $\T^{n}$. \nonumber
\end{numcases}

\begin{prop} \label{case3-stationary}
The ergodic constant $c$ is equal to $0$, and 
problem {\rm (E)} has only constant 
Lipschitz subsolutions 
$(a,a)$ for $a \in \R$.
\end{prop}
\begin{proof}
Since we can easily see that the ergodic constant 
is $0$, we only prove the second statement. 
To simplify the presentation, we argue as if $H_i$ and $v_i$
were smooth for $i=1,2$
and rigorous proof can be made by a standard
regularization argument.
Summing up the two equations in (E) and 
using (A6), we obtain
\begin{align*}
0&
\ge
H_1(Dv_1)+H_2(Dv_2)\\
&\geq H_1(0)+ DH_1(0) \cdot Dv_1 + \alpha |Dv_1|^2+
 H_2(0)+ DH_2(0) \cdot Dv_2 + \alpha |Dv_2|^2\\
 &=DH_1(0) \cdot Dv_1 + \alpha |Dv_1|^2+
 DH_2(0) \cdot Dv_2 + \alpha |Dv_2|^2.
\end{align*}
Integrate the above inequality over $\T^n$ to get
$$
0 \geq \int_{\T^n} [DH_1(0) \cdot Dv_1 + \alpha |Dv_1|^2+
 DH_2(0) \cdot Dv_2 + \alpha |Dv_2|^2] \, dx
 = \int_{\T^n} \al ( |Dv_1|^2+ |Dv_2|^2)\, dx
$$
which implies the conclusion.
\end{proof}

\begin{lem}[Monotonicity Property] \label{case3-mon}
Define 
$$
M(t):=\max_{i=1,2} \max_{x \in \T^n} u_i(x,t) \
\mbox{and }
m(t):=\min_{i=1,2} \min_{x\in \T^n} u_i(x,t).
$$
Then $t \mapsto M(t)$ is nonincreasing 
and $t \mapsto m(t)$ is nondecreasing.
\end{lem}
\begin{proof}
Fix $s \in [0,\infty)$ and let $a= m(s)$.
We have $(a,a)$ is a solution of (C2) and
$ a \geq u_i(x,s)$ for all $x \in \T^n$ and $i=1,2$.
By the comparison principle for (C2),
we have $a \geq u_i(x,t)$ for $x\in \T^n$, $t \ge s$ and $i=1,2$.
Thus $t \mapsto M(t)$ is nonincreasing.
Similarly, $t \mapsto m(t)$ is nondecreasing.
\end{proof}
By Lemma \ref{case3-mon}, we can define
$$
\ol{M}:= \lim_{t \to \infty} M(t) \
\mbox{and }
\ul{m}:=\lim_{t\to \infty} m(t).
$$
\begin{proof}[Proof of Theorem {\rm \ref{case3-thm}}]
If $\ol{M}=\ul{m}$ then we immediately get 
the conclusion and therefore we suppose 
by contradiction that $\ol{M}>\ul{m}$ 
and show the contradiction. 

Since
$\{u_i(\cdot,t)\}_{t>0}$ is compact in
$W^{1,\infty}(\T^n)$ for $i=1,2$,
there exists a sequence $T_n \to \infty$
so that $\{u_i(\cdot, T_n)\}$ converges uniformly
as $n \to \infty$ for $i=1,2$.
By the maximum principle,
$$
\|u_i(\cdot,T_n+\cdot)-u_i(\cdot, T_m+\cdot)\|_{L^\infty(\T^n \times (0,\infty))}
\leq \|u_i(\cdot, T_n)-u_i(\cdot, T_m)\|_{L^\infty(\T^n)}
$$
for $i=1,2$ and $m,n\in \N$.
Hence $\{u_i(\cdot,T_n+\cdot)\}$ is a Cauchy sequence
in $\BUC(\T^n \times[0,\infty))$ and therefore they converge 
to $u_i^\infty \in \BUC(\T^n \times[0,\infty))$ for $i=1,2$.

By a standard stability result of
the theory of viscosity solutions, 
$(u_1^\infty,u_2^\infty)$ 
is a solution 
of \eqref{eq-c3-1}, \eqref{eq-c3-2}.
Moreover for $t>0$ 
\begin{equation*}
\max_{i=1,2} \max_{x \in \T^n} u_i^\infty(x,t)=
\lim_{n\to \infty} \max_{i=1,2} \max_{x \in \T^n} u_i(x,T_n+t)
=\lim_{n \to \infty} M(T_n+t)=\ol{M},
\end{equation*}
and similarly
\begin{equation*}
\min_{i=1,2} \min_{x \in \T^n} u_i^\infty(x,t)=\ul{m}.
\end{equation*}
Let $(x_1,t_1)$ and $(x_2,t_2)$ satisfy 
$\max_{i=1,2}u_{i}^{\infty}(x_1,t_1)=\ol{M}$ 
and 
$\min_{i=1,2}u_{i}^{\infty}(x_2,t_2)=\ul{m}$. 
Without loss of generality, we assume that
$u_{1}^{\infty}(x_1,t_1)=
\max_{i=1,2}u_{i}^{\infty}(x_1,t_1)=\ol{M}$. 
By taking $0$ as a test function 
from above of 
$u_{1}^{\infty}$ at $(x_1,t_1)$ we have 
\[
u_{1}^{\infty}(x_1,t_1)-u_{2}^{\infty}(x_1,t_1)\le0 
\]
and therefore we obtain 
$u_{1}^{\infty}(x_1,t_1)=u_{2}^{\infty}(x_1,t_1)=\ol{M}$. 
Similarly we obtain 
$u_1^\infty (x_2,t_2)=u_2^\infty (x_2,t_2)=\ul{m}$. 
In particular,
\begin{equation}\label{case3-max-min}
 \max_{x \in \T^n} u_i^\infty(x,t)=M,\,  \min_{x \in \T^n} u_i^\infty(x,t)=m
\end{equation}
for $t>0$ and $i=1,2$.

On the other hand, 
we have
\begin{equation} \label{case3-sum}
(u_1^\infty+u_2^\infty)_t + H_1(Du_1^\infty)+
H_2(Du_2^\infty)=0.
\end{equation}
Integrate \eqref{case3-sum} over $\T^n$,
use (A6), and do the same way as in 
the proof of Proposition \ref{case3-stationary}
to get
\begin{align*}
0&=\dfrac{d}{dt}\int_{\T^n} (u_1^\infty+u_2^\infty)(x,t)\,dx
+\int_{\T^n} [H_1(Du_1^\infty)+H_2(Du_2^\infty)]\,dx\\
&\ge\dfrac{d}{dt}\int_{\T^n} (u_1^\infty+u_2^\infty)(x,t)\,dx
+\al \int_{\T^n} (|Du_1^\infty|^2+|Du_2^\infty|^2) \,dx\\
&\geq  \dfrac{d}{dt}\int_{\T^n} (u_1^\infty+u_2^\infty)(x,t)\,dx
+C,
\end{align*}
where the last inequality follows from
Lemma \ref{case3-poincare} below.
Thus
$$
\dfrac{d}{dt}\int_{\T^n} (u_1^\infty+u_2^\infty)(x,t)\,dx \le -C,
$$
which implies
$$
\lim_{t\to \infty} \int_{\T^n} (u_1^\infty+u_2^\infty)(x,t)\,dx=-\infty. 
$$
This contradicts \eqref{case3-max-min} and 
the proof is complete.
\end{proof}

\begin{lem}\label{case3-poincare}
There exists a constant $\beta>0$
depending only on $n, C$ such that
$$
\int_{\T^n} |Df|^2 \,dx \ge \beta
$$
for all $f \in W^{1,\infty}(\T^n)$ such that
$\|f\|_{W^{1,\infty}(\T^n)} \le C$,
$\max_{\T^n} f=1$, and
$\min_{\T^n} f=0$.
\end{lem}
\begin{proof}
We argue by contradiction.
Were the stated estimate false, 
there would exist a sequence
$\{f_m\} \subset W^{1,\infty}(\T^n)$ such that
$\|f_m\|_{W^{1,\infty}(\T^n)} \le C$,
$\max_{\T^n} f_m =1$,
$\min_{\T^n} f_m=0$, and
\begin{equation}\label{case3-lem-poi}
\int_{\T^n} |Df_m|^2 \,dx \le \dfrac{1}{m}.
\end{equation}
By Ascoli-Arzela's theorem, we may assume
there exists $f_0 \in W^{1,\infty}(\T^n)$ so that
$$
f_m \to f_0 \quad
\mbox{uniformly on } \T^n
$$
by taking a subsequence if necessary.
It is clear that $\max_{\T^n} f_0 =1$,
$\min_{\T^n} f_0=0$.

Besides, $\|f_m\|_{H^1(\T^n)} \le C$
for all $m\in \N$.
By the Rellich-Kondrachov theorem,
$$
f_m \rightharpoonup f_0 \quad
\mbox{in } H^1(\T^n)
$$
by taking a subsequence if necessary.
By \eqref{case3-lem-poi}, we obtain
$Df_0=0$ a.e. 
Thus $f_0$ is constant, which contradicts 
the fact that $\max_{\T^n} f_0 =1$,
$\min_{\T^n} f_0=0$.
\end{proof}
\begin{rem}\label{case3-rmk}
(i) Assumption (A7) is just for simplicity. 
Indeed we can always normalize 
the Hamiltonians so that they satisfy (A7)
by substituting $(u_1,u_2)$ with $(\ol{u}_1,\ol{u}_2)$,
where
$$
\begin{cases}
\ol{u}_1(x,t):=u_1(x,t)+\dfrac{H_1(0)+H_2(0)}{2}t+\dfrac{H_1(0)-H_2(0)}{2}\\
\ol{u}_2(x,t):=u_2(x,t)+\dfrac{H_1(0)+H_2(0)}{2}t
\end{cases}
\mbox{for } (x,t)\in \cQ.
$$
(ii) 
It is clear to see that we can get a similar result for systems 
with $m$-equations. 
\\
(iii) The same procedure works for the following more
general Hamiltonians
$$
H_i(x,p)=|p-\mathbf{b}_i(x)|^2 - |\mathbf{b}_i(x)|^2
$$
for $\mathbf{b}_i \in C^1(\T^n)$ with 
$\mbox{div}\, \mathbf{b}_i=0$ on $\T^n$
for $i=1,2$. 
This type of Hamiltonians is related to
the ones in some recent works on periodic homogenization
of G-equation. 
See \cite{CNS, XY} for details.
The new key observation comes from the fact that
$$
\int_{\T^n} \mathbf{b}_i(x) \cdot D\phi(x)\,dx =
-\int_{\T^n} (\mbox{div}\,\mathbf{b}_i)\phi\,dx=0
$$
for any $\phi \in W^{1,\infty}(\T^n)$.
This identity was also used in \cite{XY} to study the
existence of approximate correctors of the cell (corrector)
problem of G-equation.
The divergence free requirement on the vector fields 
$\mathbf{b}_i$ for $i=1,2$ is critical in our argument. 
In particular, it forces (E) 
to only have constant solutions $(a,a)$ for $a\in \R$.
We do not know how to remove this requirement up to now.
\end{rem}


\section{Third case}
In this section we consider the third case pointed out 
in Introduction, 
i.e., we assume that $H=H_1=H_2$ and $H$ satisfies (A1) and (A8). 
Then (C) reduces to 
\begin{numcases}
{\textrm{(C3)} \hspace{.5cm}}
(u_{1})_t + H(x,Du_{1}) + u_{1}-u_{2} = 0
& in $\Q$, \nonumber \\
(u_{2})_t + H(x,Du_{2}) + u_{2}-u_{1} = 0
& in $\Q$, \nonumber \\
u_{1}(x,0)=u_{01}(x), \ 
u_{2}(x,0)=u_{02}(x)
& 
on $\T^{n}$. \nonumber
\end{numcases}

Let $(u_{1}, u_{2})$ be the solution of (C3). 

\begin{prop}\label{prop:u1-u2}
The function $(u_{1}-u_{2})(x,t)$ converges 
uniformly to $0$ on $\T^{n}$ as $t\to\infty$. 
\end{prop}

\begin{lem}\label{lem:u1-u2}
Set $\gam(t):=\max_{x\in\T^{n}}(u_{1}-u_{2})(x,t)$. 
Then $\gam$ is a subsolution of  
\begin{equation}\label{pf:gam}
\dot{\gam}(t)+2\gam(t)=0 \ 
\textrm{in} \ (0,\infty). 
\end{equation}
\end{lem}

\begin{proof}[Proof of Lemma {\rm \ref{lem:u1-u2}}]
Let $\phi\in C^{1}((0,\infty))$ and $\tau>0$ 
be a maximum of $\gam-\phi$. 
Choose $\xi\in\T^{n}$ such that 
$\gam(\tau)=u_{1}(\xi,\tau)-u_{2}(\xi,\tau)$. 
We define the function $\Psi$ by 
\[
\Psi(x,y,t,s):=
u_{1}(x,t)-u_{2}(y,s)
-\frac{1}{2\ep^{2}}(|x-y|^{2}+(t-s)^{2})
-|x-\xi|^{2}-(t-\tau)^{2}-\phi(t). 
\]
Let $\Psi$ achieve its maximum 
at some $(\ol{x}, \ol{y}, \ol{t}, \ol{s})$. 
By the definition of viscosity solutions we have 
\begin{align*}
&
\frac{\ol{t}-\ol{s}}{\ep^{2}}+2(\ol{t}-\tau)
+\dot \phi(\ol{t})+H(\ol{x}, \frac{\ol{x}-\ol{y}}{\ep^{2}}+2(\ol{x}-\xi))
+u_{1}(\ol{x},\ol{t})-u_{2}(\ol{x},\ol{t})\le0, \\
&\frac{\ol{t}-\ol{s}}{\ep^{2}}
+H(\ol{y}, \frac{\ol{x}-\ol{y}}{\ep^{2}})
+u_{2}(\ol{y},\ol{s})-u_{1}(\ol{y},\ol{s})\ge0. 
\end{align*}
Subtracting the two inequalities above,  
we obtain 
\begin{multline}\label{pf:ineq-1}
2(\ol{t}-\tau)+\dot \phi(\ol{t})
+H(\ol{x}, \frac{\ol{x}-\ol{y}}{\ep^{2}}+2(\ol{x}-\xi))
-H(\ol{y}, \frac{\ol{x}-\ol{y}}{\ep^{2}})\\
+u_{1}(\ol{x},\ol{t})-u_{2}(\ol{x},\ol{t}) 
-(u_{2}(\ol{y},\ol{s})-u_{1}(\ol{y},\ol{s}))\le 0. 
\end{multline}

By the usual argument we may assume that  
\begin{equation}\label{pf:ep-to-0}
\ol{x}, \ol{y}\to\xi, \ 
\ol{t}, \ol{s}\to\tau, \ 
\frac{\ol{x}-\ol{y}}{\ep^{2}}\to p
\end{equation}
as $\ep\to0$ by taking a subsequence if necessary.  
Sending $\ep\to0$ in \eqref{pf:ineq-1}, we get 
\[
\dot{\phi}(\tau)+2\gam(\tau)\le0, 
\]
which is the conclusion. 
\end{proof}
\begin{proof}[Proof of Proposition {\rm \ref{prop:u1-u2}}]
Let $\gam$ be the function defined in Lemma \ref{lem:u1-u2} 
and set $C:=\|u_{01}-u_{02}\|_{L^\infty(\T^n)}$ and 
$\beta(t):=Ce^{-2t}$ for $t \in (0,\infty)$.
Then
$$
\dot \beta(t) +2 \beta(t)=0,
$$
and $\beta(0) \ge \gam(0)$.
By the comparison principle
we get $\gam(t) \le \beta(t) = Ce^{-2t}$. 
Hence $u_1(x,t)-u_2(x,t) \le C e^{-2t}$
for all $x \in \T^n$, $t \in(0,\infty)$. 
By symmetry, we get $u_2(x,t)-u_1(x,t) \le C e^{-2t}$, 
which proves the proposition. 
\end{proof}

In view of Proposition \ref{prop:u1-u2} we see that 
associated with the Cauchy problem (C3) is the ergodic problem: 
\begin{equation}\label{eq:single-ergod}
H(x,Dv(x))=c \quad \mbox{in } \T^n. 
\end{equation}
By the classical result on ergodic problems in \cite{LPV}, 
there exists a pair $(v,c)\in\W(\T^n)\times\R$ 
such that $v$ is a solution of \eqref{eq:single-ergod}.
Then $(v,v,c)$ is a solution of (E).
As in Introduction 
we normalize the ergodic constant $c$ to be $0$ 
by replacing $H$ by $H-c$.

We notice that $(v+M,v+M,0)$ is still 
a viscosity solution of (E) 
for  any $M\in\R$. 
Therefore subtracting a positive constant from $v$ if necessary, 
we may assume that 
\begin{equation}\label{bdd:u-v}
1\le u_i(x,t)-v(x)\le C 
\quad\textrm{for all} \ (x,t)\in\cQ, \ 
i=1,2 \ \textrm{and some} \ 
C>0 
\end{equation}
and we fix such a constant $C$.

We define the functions 
$\al_{\eta}^{\pm}, \beta_{\eta}^{\pm}:[0,\infty)\to\R$ by 
\begin{align}
&
\al_{\eta}^{+}(s):=\min_{x\in\T^{n}, t\ge s}
\Bigr(\frac{u_{1}(x,t)-v(x)+\eta(t-s)}{u_{1}(x,s)-v(x)}\Bigr), 
\label{def:al-plus} \\
&
\beta_{\eta}^{+}(s):=\min_{x\in\T^{n}, t\ge s}
\Bigr(\frac{u_{2}(x,t)-v(x)+\eta(t-s)}{u_{2}(x,s)-v(x)}\Bigr), 
\label{def:beta-plus}\\
&
\al_{\eta}^{-}(s):=\max_{x\in\T^{n}, t\ge s}
\Bigr(\frac{u_{1}(x,t)-v(x)-\eta(t-s)}{u_{1}(x,s)-v(x)}\Bigr), 
\nonumber\\
&
\beta_{\eta}^{-}(s):=\max_{x\in\T^{n}, t\ge s}
\Bigr(\frac{u_{2}(x,t)-v(x)-\eta(t-s)}{u_{2}(x,s)-v(x)}\Bigr)
\nonumber
\end{align}
for $\eta\in(0,\eta_{0}]$. 
By the uniform continuity of $u_{i}$ and $v$, 
we have 
$\al_{\eta}^{\pm}, \beta_{\eta}^{\pm}\in C([0,\infty))$.  
It is easy to see  that
$0\le\al_{\eta}^{+}(s),\beta_{\eta}^{+}(s)\le 1$ and 
$\al_{\eta}^{-}(s), \beta_{\eta}^{-}(s)\ge 1$ 
for all $s\in[0,\infty)$ 
and $\eta\in(0,\eta_{0}]$.

\begin{lem}[Key Lemma]\label{lem:key}
Let $C$ be the constant fixed in \eqref{bdd:u-v}. \\
{\rm (i)} 
Assume that {\rm (A8)$^{+}$} holds. 
For any $\eta\in(0,\eta_{0}]$ 
there exists $s_{\eta}>0$ such that 
the pair of the functions $(\al_{\eta}^{+},\beta_{\eta}^{+})$ is 
a supersolution of 
\begin{numcases}
{}
\max\{(\al_{\eta}^{+})^{'}(s)+\dfrac{\psi_\eta}{C}
(\al_{\eta}^{+}(s)-1)+F(\al_{\eta}^{+}(s)-\beta_{\eta}^{+}(s)), \nonumber\\
\hspace*{6cm}
\al_{\eta}^{+}(s)-1\}=0 
& in $(s_{\eta},\infty)$, \label{lem:key:1}\\
\max\{(\beta_{\eta}^{+})^{'}(s)+\dfrac{\psi_\eta}{C}
(\beta_{\eta}^{+}(s)-1)+F(\beta_{\eta}^{+}(s)-\al_{\eta}^{+}(s)), \nonumber\\
\hspace*{6cm}
\beta_{\eta}^{+}(s)-1\}=0 
& in $(s_{\eta},\infty)$, \label{lem:key:2}
\end{numcases}
where 
\[
F(r):=
\left\{
\begin{aligned}
&
C r 
&& \textrm{if} \ r\ge0, \\
&
\frac{r}{C}
&& \textrm{if} \ r<0.  
\end{aligned}
\right.
\]
{\rm (ii)} 
Assume that {\rm (A8)$^{-}$} holds. 
For any $\eta\in(0,\eta_{0}]$ 
there exists $s_{\eta}>0$ such that 
the pair of the functions $(\al_{\eta}^{-},\beta_{\eta}^{-})$ is 
a subsolution of 
\begin{numcases}
{}
\min\{(\al_{\eta}^{-})^{'}(s)+\dfrac{\psi_\eta}{C}\cdot 
\frac{\al_{\eta}^{-}(s)-1}{\al_{\eta}^{-}(s)}
+F(\al_{\eta}^{-}(s)-\beta_{\eta}^{-}(s)), \nonumber\\
\hspace*{6cm}
\al_{\eta}^{-}(s)-1\}=0 
& in $(s_{\eta},\infty)$, \label{lem:key:3}\\
\min\{(\beta_{\eta}^{-})^{'}(s)+\dfrac{\psi_\eta}{C}\cdot
\frac{\beta_{\eta}^{-}(s)-1}{\beta_{\eta}^{-}(s)}
+F(\beta_{\eta}^{-}(s)-\al_{\eta}^{-}(s)), \nonumber\\
\hspace*{6cm}
\beta_{\eta}^{-}(s)-1\}=0 
& in $(s_{\eta},\infty)$.  \label{lem:key:4}
\end{numcases}
\end{lem}

\begin{proof}
We only prove (i), since we can prove (ii) similarly. 
Fix $\mu\in(0,\eta_{0}]$. 
By abuse of notations we write $\al, \beta$ 
for $\al_{\eta}^{+}, \beta_{\eta}^{+}$. 
Recall that $\al(s), \beta(s)\leq 1$  for any $s\geq 0$. 
By Lemma \ref{lem:u1-u2}, there exists $s_{\eta}>0$ 
such that $|u_{1}(x,t)-u_{2}(x,t)|\le\eta/2$ 
for all $x \in \T^n$ and $t \ge s_\eta$.

We only consider the case where 
$(\al-\phi)(s)>(\al-\phi)(\sig)$ 
for some $\phi\in C^{1}((0,\infty))$, $\sig>s_{\eta}$, $\del>0$ 
and all $s\in[\sig-\del,\sig+\del]\setminus\{\sig\}$,  
since a similar argument holds for $\beta$. 
Since there is nothing to check in the case where
$\al(\sig)=1$, we assume that $\al(\sig)<1$. 
We choose $\xi\in\T^n$ 
and $\tau\ge\sig$ 
such that 
\[
\al(\sig)=
\frac{u_{1}(\xi,\tau)-v(\xi)+\eta(\tau-\sig)}
{u_{1}(\xi,\sig)-v(\xi)}
=:\frac{\al_{2}}{\al_{1}}. 
\]
We write $\al$ 
for $\al(\sig)$ henceforth.

Set $K:=\T^{3n}\times\{(t,s)\mid t\ge s, 
s\in[\sig-\del,\sig+\del]\}$. 
For $\ep\in(0,1)$, 
we define the function 
$\Psi:K\to\R$ by 
\begin{align*}
{}&\Psi(x,y,z,t,s)\\
:=&
\frac{u_{1}(x,t)-v(z)+\eta(t-s)
}{u_{1}(y,s)-v(z)}-\phi(s)
+\frac{1}{2\ep^{2}}(|x-y|^{2}+|x-z|^{2})
+|x-\xi|^{2}+(t-\tau)^{2}. 
\end{align*}

Let $\Psi$ achieve its minimum over 
$K$ at some $(\ol{x}, \ol{y}, \ol{z}, \ol{t}, \ol{s})$. 
Set 
\begin{align*}
\ol{\al}_{1}&
:=u_{1}(\ol{y},\ol{s})-v(\ol{z}), \ 
\ol{\al}_{2}=
u_{1}(\ol{x},\ol{t})-v(\ol{z})+\eta(\ol{t}-\ol{s}), \ 
\ol{\al}:=\frac{\ol{\al}_{2}}{\ol{\al}_{1}}, \\
\ol{p}&:=\frac{\ol{y}-\ol{x}}{\ep^{2}} \ 
\textrm{and} \ 
\ol{q}:=\frac{\ol{z}-\ol{x}}{\ep^{2}}. 
\end{align*}

We have by the definition of viscosity solutions
\begin{equation}\label{pf:main-lem:ineq-1}
\left\{
\begin{aligned}
-\eta-2 \ol{\al}_{1}(\ol{t}-\tau)
+H(\ol{x}, D_{x}u_{1}(\ol{x},\ol{t}))
+(u_{1}-u_{2})(\ol{x},\ol{t})\ge0, \\
-\frac{1}{\ol{\al}}(\eta+\ol{\al}_{1}\phi^{'}(\ol{s}))
+H(\ol{y}, D_{y}u_{1}(\ol{y},\ol{s}))
+(u_{1}-u_{2})(\ol{y},\ol{s})\le0, \\
H(\ol{z}, D_{z}v(\ol{z}))
\le0, 
\end{aligned}
\right. 
\end{equation}
where 
\begin{align*}
&D_{x}u_{1}(\ol{x},\ol{t})
=
\ol{\al}_{1}\bigl\{
\ol{p}+\ol{q} 
+2(\xi-\ol{x})\bigr\}, \\
&D_{y}u_{1}(\ol{y},\ol{s}) 
=
\frac{\ol{\al}_{1}}{\ol{\al}}\ol{p}, \\
&D_{z}v(\ol{z})
=
\frac{\ol{\al}_{1}}{1-\ol{\al}}
\ol{q}. 
\end{align*}

By taking a subsequence if necessary, 
we may assume that 
\[\ol{x}, \ol{y}, \ol{z}\to \xi\ 
\textrm{and} \ 
\ol{t} \to \tau, \ 
\ol{s}\to\sig \ 
\textrm{as} \ \ep\to0. 
\]
Since $u_{i}$, $v$ are Lipschitz continuous, 
we have 
\[
\frac{|\ol{x}-\ol{y}|}{\ep^{2}}
+\frac{|\ol{x}-\ol{z}|}{\ep^{2}}\le M 
\]
for some $M>0$ and all $\ep\in(0,1)$.
We may assume that 
\[
\ol{p}:=\frac{\ol{y}-\ol{x}}{\ep^{2}}\to p, \ 
\ol{q}:=\frac{\ol{z}-\ol{x}}{\ep^{2}}\to q
\]
as $\ep\to0$ for some $p,q\in B(0,M)$.

Sending $\ep\to0$ in \eqref{pf:main-lem:ineq-1} 
yields 
\begin{align}
-\eta+H(\xi, \tilde{P}+Q)
+(u_{1}-u_{2})(\xi,\tau)\ge0, 
\nonumber\\
-\frac{1}{\al(\sig)}(\eta+\al_{1}\phi^{'}(\sig))
+H(\xi, P)
+(u_{1}-u_{2})(\xi,\sig)\le0, 
\label{pf:main-lem:ineq-3}\\
H(\xi, Q)\le0, \nonumber
\end{align}
where 
\[
P:=\frac{\al_{1}}{\al}p, \ 
Q:=\frac{\al_{1}}{1-\al}q, \
\tilde P = \al(P-Q).
\]
Recalling that $(u_{1}-u_{2})(\xi,\tau)\le\eta/2$, 
we have 
\[
H(\xi, \tilde{P}+Q)\ge 
\eta/2. 
\]
Therefore, by using (A8)$^{+}$, we obtain 
\begin{equation}\label{pf:main-lem:ineq-5}
H(\xi, \tilde{P}+Q)
\le\, 
\al
H(\xi, P)
-\psi_{\eta}(1-\al) 
\end{equation}
for some $\psi_{\eta}>0$.

Noting that 
\[
\beta(\sig)
\le 
\frac{u_{2}(\xi,\tau)-v(\xi)+\eta(\tau-\sig)}
{u_{2}(\xi,\sig)-v(\xi)}
=:\frac{\beta_{2}}{\beta_{1}}, 
\]
we calculate that 
\begin{align*}
{}&
(u_{1}-u_{2})(\xi,\tau)
-\al(u_{1}-u_{2})(\xi,\sig) \\
=&\, 
-(u_{2}(\xi,\tau)-v(\xi)+\eta(\tau-\sig))
+\al(u_{2}(\xi,\sig)-v(\xi))\\
=&\, 
-\beta_{1}
\bigl(
\frac{\beta_{2}}{\beta_{1}}-\al
\bigr)\\
\le&\, 
-\beta_{1}(\beta(\sig)-\al(\sig)). 
\end{align*}

Therefore by 
\eqref{pf:main-lem:ineq-5} and 
\eqref{pf:main-lem:ineq-3}, 
\begin{align*}
\eta\le 
&H(\xi, \tilde{P}+Q)+(u_{1}-u_{2})(\xi,\tau)\\
\le&\, 
\al
\Bigl(
\frac{1}{\al}(\eta+\al_{1}\phi^{'}(\sig))
-(u_{1}-u_{2})(\xi,\sig)
\Bigr)
-\psi_{\eta}(1-\al)+(u_{1}-u_{2})(\xi,\tau) \\
\le&\, 
\eta+\al_{1}\phi^{'}(\sig)-\psi_{\eta}(1-\al)
+\beta_1(\al(\sig)-\beta(\sig)), 
\end{align*}
which implies 
\[
\phi^{'}(\sig)+\dfrac{\psi_\eta}{C}
(\al(\sig)-1)+\dfrac{\beta_1}{\al_1}(\al(\sig)-\beta(\sig))
\ge 0.
\]
Combining the above inequality with
the fact that $1/C\le\beta_{1}/\al_{1}\le C$,
we have
\[
\phi^{'}(\sig)+\dfrac{\psi_\eta}{C}
(\al(\sig)-1)+F(\al(\sig)-\beta(\sig))
\ge 0.
\]
\end{proof}

\begin{lem}\label{lem:al-beta}\ \\
{\rm (i)} 
Assume that {\rm (A8)$^{+}$} holds. 
The functions $\al^{+}_{\eta}(s)$ and $\beta^{+}_{\eta}(s)$ converge to $1$ as $s\to \infty$ for each $\eta\in(0,\eta_{0}]$.\\
{\rm (ii)} 
Assume that {\rm (A8)$^{-}$} holds. 
The functions $\al^{-}_{\eta}(s)$ and $\beta^{-}_{\eta}(s)$ converge to $1$ as $s\to \infty$ for each $\eta\in(0,\eta_{0}]$.
\end{lem}

\begin{proof}
Fix $\eta\in(0,\eta_{0}]$. 
We first recall that, by definition, 
\[
\al_{\eta}^{+}(s) \leq 1 \leq \al_{\eta}^{-}(s), 
\
\beta_{\eta}^{+}(s) \leq 1 \leq \beta_{\eta}^{-}(s)
\]
for any $s\geq 0$. 
On the other hand, one checks easily that the pairs
$$ 
\bigl(1 +  (\gamma_1-1) \exp(-\frac{\psi_{\eta}}{C}t),  
1 +  (\gamma_1-1) \exp(-\frac{\psi_{\eta}}{C}t)\bigl)
$$
and
$$ 
\bigl(1 +  (\gamma_2-1) \exp(-\frac{\psi_{\eta}}{C\gamma_2}t), 
1 +  (\gamma_2-1) \exp(-\frac{\psi_{\eta}}{C\gamma_2}t)\bigr)
$$
are, respectively, a subsolution and a supersolution of 
\eqref{lem:key:1}-\eqref{lem:key:2} and 
\eqref{lem:key:3}-\eqref{lem:key:4} 
for $\gamma_1 = \min\{\al_\eta^+(0),\beta_\eta^+(0)\}$,
and $\gamma_2=\max\{\al_{\eta}^{-}(0),\beta_{\eta}^{-}(0)\}$.
Therefore, by the comparison principle in \cite{LEN88, EL, IK1} , 
we get
$$ 
\al_{\eta}^{+}(s), \beta_{\eta}^{+}(s) \geq
1 +  (\gamma_1-1) \exp(-\frac{\psi_{\eta}}{C}t)
$$
and
$$ 
\al_{\eta}^{-}(s), \beta_{\eta}^{-}(s) \leq
1 +  (\gamma_2-1) \exp(-\frac{\psi_{\eta}}{C\gamma_2}t),
$$
which give us the conclusion.
\end{proof}

By Lemma \ref{lem:al-beta}, we immediately get
the following proposition. 

\begin{prop}[Asymptotically Monotone Property]\label{lem:asymp-mono}\ \\
{\rm (i)} 
{\bf (Asymptotically Increasing Property)} \\
Assume that {\rm(A8)}$^{+}$ holds. 
For $\eta \in (0,\eta_0]$, there exists a function
$\delta_\eta:[0,\infty) \to [0,1]$ such that
$$
\lim_{s \to \infty} \delta_\eta(s)=0 
$$
and
$$
u_i(x,s)-u_i(x,t)-\eta(t-s)  \le \delta_\eta(s) 
$$
for all $x \in \T^n$, $t\ge s \ge 0$ and $i=1,2$.\\
{\rm (ii)} 
{\bf (Asymptotically Decreasing Property)} \\
Assume that {\rm(A8)}$^{-}$ holds. 
For $\eta \in (0,\eta_0]$, there exists a function
$\delta_\eta:[0,\infty) \to [0,1]$ such that
$$
\lim_{s \to \infty} \delta_\eta(s)=0 
$$
and
$$
u_i(x,t)-u_i(x,s)-\eta(t-s)  \le \delta_\eta(s),
$$
for all $x \in \T^n$, $t\ge s \ge 0$ and $i=1,2$.
\end{prop}

Theorem \ref{thm:case2} is a straightforward result of 
the above proposition. 
See \cite[Section 4]{BS} or \cite[Section 5]{BM} 
for the detail.

Finally we remark that if we want to deal, at the same time, 
with the Hamiltonians of the form
\[
H(x,p):=|p|-f(x), 
\]
we can generalize Theorem \ref{thm:case2} 
as in \cite{BS}. 
We replace (A8) by 
\begin{itemize}
\item[{\rm(A9)}] 
Either of the following assumption 
{\rm (A9)$^{+}$} or {\rm (A9)$^{-}$} holds: 

\item[{\rm(A9)}$^{+}$] 
There exists a closed set $K\subset\T^n$ 
($K$ is possibly empty) having the properties\\
\begin{itemize}
\item[{\rm (i)}] 
$\min_{p\in\R^{n}}H(x,p)=0$ for all $x\in K$, \\
\item[{\rm (ii)}] 
for each $\ep>0$ there exists a modulus $\psi_{\ep}(r)>0$ 
for all $r>0$ and $\eta_{0}^{\ep}>0$ such that 
for all $\eta\in(0,\eta_{0}^{\ep}]$ 
if $\dist(x,K)\ge\ep$, $H(x,p+q)\ge\eta$ and 
$H(x,q)\le0$ for some 
$x\in\T^n$ and $p,q\in\R$, then for any $\mu\in(0,1]$, 
\[
\mu H(x,\frac{p}{\mu}+q)\ge 
H(x,p+q)+\psi_{\ep}(\eta)(1-\mu). 
\]
\end{itemize}
\item[{\rm(A9)}$^{-}$] 
There exists a closed set $K\subset\T^n$ 
($K$ is possibly empty) having the properties\\
\begin{itemize}
\item[{\rm (i)}] 
$\min_{p\in\R^{n}}H(x,p)=0$ for all $x\in K$, \\
\item[{\rm (ii)}] 
for each $\ep>0$ there exists a modulus $\psi_{\ep}(r)>0$ 
for all $r>0$ and $\eta_{0}^{\ep}>0$ such that 
for all $\eta\in(0,\eta_{0}^{\ep}]$ 
if $\dist(x,K)\ge\ep$, $H(x,p+q)\le-\eta$ and 
$H(x,q)\ge0$ for some 
$x\in\T^n$ and $p,q\in\R$, then for any $\mu\in(0,1]$, 
\[
\mu H(x,\frac{p}{\mu}+q)\le 
H(x,p+q)-\frac{\psi_{\ep}(\eta)(\mu-1)}{\mu}. 
\]
\end{itemize}
\end{itemize}

\begin{thm}\label{thm:case2-general}
The result of Theorem {\rm \ref{thm:case2}} still
holds if we replace {\rm (A8)} by {\rm (A9)}. 
\end{thm}
\begin{proof}[Sketch of Proof]
By the argument same as in the proof of Propositions 
\ref{prop:mon1}, \ref{prop:mon2} 
we can see $(u_{1}+u_2)|_{K}$ 
and $\max\{u_{1}, u_{2}\}|_{K}$ are nonincreasing 
and therefore we see that $u_{i}$ converge uniformly on $K$ 
as $t\to\infty$ for $i=1,2$. 

Setting $K_{\ep}:=\{x\in\T^{n}\mid d(x,K)\ge\ep\}$, 
we see that $u_{i}$ are asymptotically monotone 
on $\R^{n}\setminus K_{\ep}$ for every $\ep>0$, 
which implies that $u_{i}$ converges uniformly on $\R^{n}\setminus K$ 
as $t\to\infty$ for $i=1,2$ as in \cite{BS}. 
\end{proof}


\section{Appendix}
We present a sketch 
of the proof based on Proposition \ref{prop:cell} 
from \cite{CGT2} for the reader's convenience.

\begin{proof}[Sketch of the proof of Proposition {\rm \ref{prop:cell}}]
Without loss of generality, we may assume $c_1=c_2=1$.
The existence of $(v_1,v_2,\ol{H}_1,\ol{H}_2)$
can be proved by repeating the argument same as in the
first part of Proposition \ref{prop:case1:erg}.
We here only prove that $\ol{H}_1 + \ol{H}_2$ is unique.

Suppose by contradiction
that there exist two pairs $(\lambda_1, \lambda_2) \in \mathbb R^2$ 
and $(\mu_1, \mu_2) \in \mathbb R^2$ 
such that $\lambda_1+\lambda_2<\mu_1+\mu_2$ 
and two pair of continuous functions 
$(v_1, v_2)$, $(\ol{v}_1,\ol{v}_2)$
such that $(v_1, v_2)$, $(\ol{v}_1,\ol{v}_2)$ 
are viscosity solutions of the following systems 
\begin{equation}
\left\{ \begin{aligned}
H_1(x,Dv_1)+v_1-v_2&=\lambda_1 \vspace{.05in}\\
H_2(x,Dv_2)+v_2-v_1&=\lambda_2 \\
\end{aligned} \right. 
\quad \mbox{in}~ \mathbb T^n,
\notag
\end{equation}
and
\begin{equation}
\left\{ \begin{aligned}
H_1(x,D\ol{v}_1)+\ol{v}_1-\ol{v}_2&=\mu_1 \vspace{.05in}\\
H_2(x,D\ol{v}_2)+\ol{v}_2-\ol{v}_1&=\mu_2 \\
\end{aligned} \right. 
\quad \mbox{in}~ \mathbb T^n,
\notag
\end{equation}
respectively.

For a suitably large constant $C>0$, 
$\displaystyle (v_{1}+\dfrac{\lam_2-\lam_1}{2}
-\dfrac{\lam_{1}+\lam_2}{2}t-C, v_{2}-\dfrac{\lam_{1}+\lam_2}{2}t-C)$ 
and 
$\displaystyle (\ol{v}_{1}+\dfrac{\mu_2-\mu_1}{2}
-\dfrac{\mu_{1}+\mu_2}{2}t+C, \ol{v}_{2}-\dfrac{\mu_{1}+\mu_2}{2}t+C)$
are respectively
a subsolution and a supersolution of (C).
By the comparison principle for (C), 
Proposition \ref{prop:comp-for-C}, we obtain particularly
$$
v_{1}+\dfrac{\lam_2-\lam_1}{2}
-\dfrac{\lam_{1}+\lam_2}{2}t-C \le
\ol{v}_{1}+\dfrac{\mu_2-\mu_1}{2}
-\dfrac{\mu_{1}+\mu_2}{2}t+C,\quad
\mbox{in } \cQ
$$
which contradicts the fact that 
$\lambda_1+\lambda_2<\mu_1+\mu_2$. 
\end{proof}


\noindent
\textbf{Acknowledgements. }
The authors are grateful to 
Professors Lawrence C. Evans and Fraydoun Rezakhanlou
for useful comments, 
as well as to Professor Olivier Ley for 
personal communications and sending to the authors 
his latest manuscript with Professors F. Camilli, P. Loreti, 
and V. Nguyen before its publication. 
This work was done while the first author 
visited Mathematics Department, 
University of California, Berkeley.


%

\bibliographystyle{amsplain}
\providecommand{\bysame}{\leavevmode\hbox to3em{\hrulefill}\thinspace}
\providecommand{\MR}{\relax\ifhmode\unskip\space\fi MR }
\providecommand{\MRhref}[2]{%
  \href{http://www.ams.org/mathscinet-getitem?mr=#1}{#2}
}
\providecommand{\href}[2]{#2}

\end{document}